\documentclass[eqnumis,pdf]{hjms}

\usepackage{amsfonts}
\usepackage{amsmath}
\usepackage{amssymb}
\usepackage{footnote}
\usepackage{multicol}
\usepackage{booktabs}
\usepackage[flushleft]{threeparttable}
\usepackage[font=small,labelfont=bf]{caption}
\hypersetup{
  pdftitle={Soft bitopological groups via soft elements},
  pdfauthor={S. Ray},
  pdfkeywords={Soft set, soft element, soft bitopological group, paratopological group, box topology}
}

\newcommand{\SE}{\mathrm{SE}}
\newcommand{\sse}{\subseteq_s}
\newcommand{\seq}{=_s}
\newcommand{\Sunion}{\cup_s}
\newcommand{\Sinter}{\cap_s}
\newcommand{\Pow}{\mathcal P}
\newenvironment{romanlist}{\begin{enumerate}}{\end{enumerate}}
\newenvironment{alphlist}{\begin{enumerate}}{\end{enumerate}}

\begin{document}

\markboth{S. Ray}{Soft bitopological groups via soft elements}

\title{Soft bitopological groups via soft elements}{Soft bitopological groups via soft elements}

\author{S. Ray\coraut$^{1}$}

\address{$^1$Department of Mathematics, Siksha-Bhavana (Institute of Science), Visva-Bharati, Santiniketan 731235, West Bengal, India}
\emails{subhasis.ray@visva-bharati.ac.in (S. Ray)}
\blfootnote{\, ORCID: 0000-0003-1100-4632}

\maketitle

\begin{abstract}
Let $F$ be a soft group over a parameter set $A$. We equip the selector group $\SE(F)=\prod_{t\in A}F(t)$ with the topology generated by componentwise open boxes. This construction replaces a sectionwise rule that does not define a topology on arbitrary selector subsets. It depends only on the component topologies and may therefore lose correlations between parameters. A soft bitopological group is then a soft group with two selector topologies, each making $\SE(F)$ a topological group. We also study a soft paratopological group together with its inverse topology. In this conjugate pair, inversion interchanges the two selector topologies, and equality of the two is equivalent to continuity of inversion. We prove componentwise separation criteria, finite-parameter compactness and $\omega$-boundedness results, and infinite-parameter counterexamples caused by the box topology. We also characterize when a soft union creates mixed selectors outside the two original selector sets.
\end{abstract}


\section{Introduction}
Soft set theory was introduced by Molodtsov \cite{Molodtsov1999}, and its basic operations were developed by Maji--Biswas--Roy \cite{Maji2003}. Akta\c{s} and \c{C}a\u{g}man introduced soft groups \cite{Aktas2007}. Soft topologies were studied by \c{C}a\u{g}man--Karata\c{s}--Engino\u{g}lu \cite{Cagman2011} and Shabir--Naz \cite{Shabir2011}. Work on soft topological groups includes Nazmul--Samanta \cite{Nazmul2014,Nazmul2012,Nazmul2010}, Shah--Shaheen \cite{Shah2014}, and Goldar--Ray \cite{GoldarRay2022STG}.

A soft element of $F:A\to\Pow(X)$ is a choice function $a$ with $a(t)\in F(t)$ for every parameter. Thus $\SE(F)$ is an ordinary product. Ray--Goldar used this observation for soft groups \cite{RayGoldar2017}, and Goldar--Ray introduced a selector approach to soft topology \cite{GoldarRay2019,GoldarRay2017Topology}. Their sectionwise description motivated the present construction, but open coordinate projections of arbitrary selector subsets do not form a topology. We therefore use the topology generated by selector boxes. Remarks~\ref{rem:canonical_nonrep} and~\ref{rem:selector_forgets_correlations} explain the correction and its limitation.

Classical bitopology studies a set with two topologies \cite{Kelly1963}. Soft bitopological spaces were considered by Ittanagi \cite{Ittanagi2014} and Kandil et al. \cite{Kandil2017Connected,Kandil2016}. A later soft-element treatment is given in \cite{SoftBTSviaSE}. The Hausdorff clause used there requires sectionwise disjoint soft sets and is stronger than the selector-space condition. Section~\ref{sec:separation} replaces it with disjoint selector neighbourhoods and proves the corresponding componentwise criterion.

The paper treats two forms of group bitopology. In the first, a soft group carries two soft topologies, and each selector topology makes $\SE(F)$ a topological group. The elementary results on translations, inversion, separation, and homomorphisms then follow from ordinary topological-group arguments. In the second form, a soft paratopological group is paired with its inverse topology. This conjugate topology records the possible failure of inversion to be continuous; see Mar{\'i}n--Romaguera \cite{MarinRomaguera1996}, Ravsky \cite{Ravsky2002,Ravsky2001}, Arhangel'skii--Tkachenko \cite{ArhangelskiiTkachenko2008}, and Tkachenko \cite{Tkachenko2014}.

The parameter set affects the selector topology. Since $\SE(F)=\prod_{t\in A}F(t)$ carries a box topology, finite parameter sets behave like finite products. For infinite parameter sets, compactness and $\omega$-boundedness may fail even when every coordinate is compact. Soft unions also need care: they may create mixed selectors outside the union of the original selector sets.

Soft topologies also admit ordinary representations on $A\times X$ \cite{Alcantud2021,Matejdes2016,ShiPang2015}. This equivalence is background, but it is not the same as the selector construction on $\prod_{t\in A}F(t)$. The selector topology depends only on the component topologies and can forget correlations between parameters. The main points of the paper are the corrected box construction and selector-union theorem, the conjugate topology and its automatic-continuity criteria, and the finite/infinite parameter difference for box compactness.

\section{Preliminaries}

Throughout, $A$ denotes a fixed nonempty set of parameters and $X$ denotes a nonempty universe. For a set $Y$, $\Pow(Y)$ denotes the power set of $Y$.

\subsection*{Soft sets and soft elements}
\begin{definition}[Soft set and soft element]\label{def:softset}
A \emph{soft set} over $X$ (with parameter set $A$) is a mapping $F:A\to \Pow(X)$. We assume $F(t)\neq \varnothing$ for all $t\in A$ when we speak about soft elements.

A mapping $a:A\to X$ is a \emph{soft element} of $F$ if $a(t)\in F(t)$ for all $t\in A$. In this case we write $a\in_s F$. The set of all soft elements of $F$ is denoted by
\begin{equation}
\SE(F):=\{a:A\to X:\ a(t)\in F(t)\ \forall t\in A\}.
\end{equation}
Two soft elements $a,b\in \SE(F)$ are \emph{soft equal}, written $a\seq b$, if $a(t)=b(t)$ for all $t\in A$.
\end{definition}

\begin{remark}\label{rem:product_view}
As sets one has the identification $\SE(F)=\prod_{t\in A}F(t)$. If $A$ is infinite, the axiom of choice is typically used to guarantee $\SE(F)\neq\varnothing$ whenever each $F(t)\neq\varnothing$.
\end{remark}

\begin{definition}[Soft subset, union and intersection]\label{def:soft_ops}
Let $F,H:A\to\Pow(X)$ be soft sets. We say $H$ is a \emph{soft subset} of $F$, written $H\sse F$, if $H(t)\subseteq F(t)$ for all $t\in A$. 

The \emph{soft union} and \emph{soft intersection} are defined parameterwise by
\begin{equation}
(F\Sunion H)(t):=F(t)\cup H(t),\qquad (F\Sinter H)(t):=F(t)\cap H(t)\qquad (t\in A).
\end{equation}
The \emph{empty soft set} $\Phi$ is given by $\Phi(t)=\varnothing$ for all $t\in A$.
\end{definition}

\begin{remark}\label{rem:SE_union_intersection}
In general one only has
\begin{equation}
\SE(F)\cup \SE(H)\subseteq \SE(F\Sunion H),
\end{equation}
and the inclusion may be strict because a soft element of $F\Sunion H$ may choose values from $F(t)$ for some parameters and from $H(t)$ for other parameters \cite{GoldarRay2019}. 
For intersections one always has
\begin{equation}
\SE(F\Sinter H)=\SE(F)\cap \SE(H),
\end{equation}
as observed by Goldar--Ray \cite{GoldarRay2019}.
\end{remark}

\begin{theorem}\label{thm:selector_union_exact}
Let $F,H:A\to\Pow(X)$ be soft sets with nonempty sections. Put
\begin{equation}
I_{H\setminus F}:=\{t\in A: H(t)\setminus F(t)\neq\varnothing\},\qquad
I_{F\setminus H}:=\{t\in A: F(t)\setminus H(t)\neq\varnothing\}.
\end{equation}
Then
\begin{equation}
\SE(F)\cup \SE(H)\subseteq \SE(F\Sunion H)=\prod_{t\in A}(F(t)\cup H(t)).
\end{equation}
Moreover, the inclusion is strict if and only if there exist distinct parameters
$s,t\in A$ such that $H(s)\setminus F(s)\neq\varnothing$ and $F(t)\setminus H(t)\neq\varnothing$.
Equivalently, equality holds in exactly the following cases:
\begin{romanlist}
\item $H(t)\subseteq F(t)$ for every $t\in A$;
\item $F(t)\subseteq H(t)$ for every $t\in A$;
\item all strict differences in the two directions occur at one and the same parameter.
\end{romanlist}
In particular, if two opposite strict differences occur at two different parameters, then $\SE(F\Sunion H)$ contains mixed selectors not lying in $\SE(F)\cup\SE(H)$.
\end{theorem}

\begin{proof}
The displayed inclusion is immediate from the product representation. If $s\neq t$, choose $x_s\in H(s)\setminus F(s)$ and $x_t\in F(t)\setminus H(t)$, and choose values from $F(r)\cup H(r)$ at all other parameters. The resulting selector belongs to $\SE(F\Sunion H)$, but it is not in $\SE(F)$ because of the coordinate $s$, and it is not in $\SE(H)$ because of the coordinate $t$.

Conversely, if a selector $a\in\SE(F\Sunion H)$ is not in $\SE(F)\cup\SE(H)$, then $a(s)\notin F(s)$ for some $s$ and $a(t)\notin H(t)$ for some $t$. Since $a(s)\in F(s)\cup H(s)$ and $a(t)\in F(t)\cup H(t)$, this gives $a(s)\in H(s)\setminus F(s)$ and $a(t)\in F(t)\setminus H(t)$. If $s=t$, then $a(s)$ would lie in both $H(s)\setminus F(s)$ and $F(s)\setminus H(s)$, impossible. Thus $s\neq t$. The final characterization is just the negation of this condition.
\end{proof}

\subsection*{Soft groups via soft elements}
Let $(G,\cdot)$ be a group with identity element $e$.

\begin{definition}[Soft group]\label{def:soft_group}
A soft set $F:A\to \Pow(G)$ is called a \emph{soft group over $G$} if each $F(t)$ is a subgroup of $G$ (Akta\c{s}--\c{C}a\u{g}man \cite{Aktas2007}). Equivalently (classical viewpoint), $\SE(F)$ is closed under pointwise multiplication and inversion (Ray--Goldar \cite{RayGoldar2017}):
\begin{equation}
(a\ast b)(t):=a(t)\cdot b(t),\qquad a^{-1}(t):=a(t)^{-1}\qquad (t\in A),
\end{equation}
for all $a,b\in\SE(F)$. In this case $(\SE(F),\ast)$ is a group, called the \emph{soft-element group} of $F$.
\end{definition}

\begin{remark}\label{rem:soft_identity}
The identity in $(\SE(F),\ast)$ is the constant soft element $e_F$ defined by $e_F(t)=e$ for all $t\in A$.
\end{remark}

\begin{example}\label{ex:Z_2_3}
Let $G=(\mathbb Z,+)$ and $A=\{t_1,t_2\}$. Define $F(t_1)=2\mathbb Z$ and $F(t_2)=3\mathbb Z$. Then $F$ is a soft group and $\SE(F)=2\mathbb Z\times 3\mathbb Z$ with coordinatewise addition.
\end{example}

\subsection*{Soft topologies and selector topologies}
\begin{definition}[Soft topology]\label{def:soft_topology}
Let $F:A\to\Pow(X)$ be a soft set with nonempty sections. A family $\tau$ of soft subsets of $F$ is a \emph{soft topology} on $F$ if:
\begin{alphlist}
\item $\Phi\in \tau$ and $F\in \tau$;
\item $\tau$ is closed under arbitrary soft unions;
\item $\tau$ is closed under finite soft intersections.
\end{alphlist}
Then $(F,\tau)$ is a \emph{soft topological space}; members of $\tau$ are \emph{soft open} sets \cite{Cagman2011,Shabir2011}. For each $t\in A$, the family
\begin{equation}
\tau_t:=\{H(t):H\in\tau\}
\end{equation}
is a topology on $F(t)$, called the \emph{component topology}.
\end{definition}

\begin{definition}[Sections and selector boxes]\label{def:section_of_subset}
Let $T\subseteq \SE(F)$. Its section at $t\in A$ is
\begin{equation}
T(t):=\{a(t):a\in T\}\subseteq F(t).
\end{equation}
If $U_t\subseteq F(t)$ for every $t\in A$, write
\begin{equation}
\Box_{t\in A}U_t:=\prod_{t\in A}U_t\subseteq\SE(F).
\end{equation}
This set is a \emph{selector box}. For a soft subset $H\sse F$,
\begin{equation}
\SE(H)=\Box_{t\in A}H(t).
\end{equation}
\end{definition}

\begin{definition}[Selector box and product topologies]\label{def:box_tych}
Let $(F,\tau)$ be a soft topological space. The boxes
\begin{equation}
\mathcal B_\tau:=\left\{\Box_{t\in A}U_t:U_t\in\tau_t\text{ for every }t\in A\right\}
\end{equation}
form a base on $\SE(F)$. The topology generated by this base is denoted by $\tau^\ast=\tau^\Box$ and is called the \emph{selector box topology}. We write $\tau^\Pi$ for the Tychonoff product topology, whose basic boxes satisfy $U_t=F(t)$ for all but finitely many $t$.
\end{definition}

\begin{proposition}\label{thm:induced_topology}
Let $(F,\tau)$ be a soft topological space. Then
\begin{equation}
\tau^\Pi\subseteq\tau^\ast,
\end{equation}
and equality holds when $A$ is finite. If
\begin{equation}
\tau_{\mathrm{can}}:=\{H\sse F:H(t)\in\tau_t\text{ for every }t\in A\}
\end{equation}
is the canonical componentwise enlargement of $\tau$, then
\begin{equation}
\{\SE(H):H\in\tau_{\mathrm{can}}\}=\mathcal B_\tau.
\end{equation}
\end{proposition}

\begin{proof}
The product topology uses only boxes with finitely many nontrivial coordinates, whereas $\tau^\ast$ uses all boxes. The two bases agree when $A$ is finite. The last equality follows from $\SE(H)=\prod_tH(t)$.
\end{proof}

\begin{remark}\label{rem:canonical_nonrep}
Goldar and Ray introduced the selector viewpoint for soft topology \cite{GoldarRay2017Topology}. Their sectionwise rule needs correction for arbitrary subsets of a product. In general,
\begin{equation}
(T_1\cap T_2)(t)\neq T_1(t)\cap T_2(t).
\end{equation}
For example, let $A=\{1,2\}$ and
\begin{equation}
T_1=\{(0,0),(1,1)\},\qquad T_2=\{(0,0),(1,2)\}.
\end{equation}
The first section of $T_1\cap T_2$ is $\{0\}$, while the intersection of the first sections is $\{0,1\}$. Thus open sections of an arbitrary selector subset do not define a topology. The box base in Definition~\ref{def:box_tych} avoids this problem.
\end{remark}

\begin{remark}\label{rem:selector_forgets_correlations}
The map $\tau\mapsto\tau^\ast$ depends only on the component topologies. It need not be injective, and
\begin{equation}
\tau^\ast=(\tau_{\mathrm{can}})^\ast.
\end{equation}
For a simple example, let $A=\{1,2\}$, let $F(1)=F(2)=\{0,1\}$, and define $H_i(t)=\{i\}$ for $i=0,1$ and both parameters. Then
\begin{equation}
\tau=\{\Phi,F,H_0,H_1\}
\end{equation}
is a soft topology whose two component topologies are discrete. Hence $\tau^\ast$ is the discrete topology on $\{0,1\}^2$, although $\tau$ itself has only two nontrivial soft open sets. In particular, the selectors $(0,1)$ and $(1,0)$ are not separated by a proper member of $\tau$. The selector topology therefore records the component topologies but not their original correlation.
\end{remark}

\subsection*{Soft mappings and soft continuity}
\begin{definition}[Soft mapping]\label{def:soft_mapping}
Let $F$ and $H$ be soft sets. In this paper, a \emph{soft mapping} from $F$ to $H$ is a function
\begin{equation}
f:\SE(F)\longrightarrow\SE(H),
\end{equation}
and we write $f:F\to_s H$. If $W\sse F$ and $V\sse H$, then $f[\SE(W)]\subseteq\SE(H)$ and $f^{-1}[\SE(V)]\subseteq\SE(F)$ are ordinary selector subsets; they are not assumed to be soft sets.
\end{definition}

\begin{definition}[Soft continuity]\label{def:soft_cont}
Let $(F,\tau)$ and $(H,\sigma)$ be soft topological spaces. A soft mapping $f:F\to_s H$ is \emph{soft continuous} if
\begin{equation}
f:(\SE(F),\tau^\ast)\longrightarrow(\SE(H),\sigma^\ast)
\end{equation}
is continuous. It is enough to test inverse images of the basic selector boxes $\SE(V)$ with $V\in\sigma_{\mathrm{can}}$.
\end{definition}

\begin{remark}
Goldar--Ray \cite{GoldarRay2017Topology} introduced continuity through soft elements. Definition~\ref{def:soft_cont} uses the corrected box topology and does not identify an arbitrary selector subset with a soft set.
\end{remark}

\subsection*{Soft topological groups and induced bitopologies}
We adapt the selector-based definition of Goldar--Ray \cite{GoldarRay2022STG} to the corrected box topology.

\begin{definition}[Soft topological group]\label{def:soft_top_group}
Let $F$ be a soft group over a group $G$, and let $\tau$ be a soft topology on $F$. The pair $(F,\tau)$ is called a \emph{soft topological group} if the induced topological group $(\SE(F),\ast,\tau^\ast)$ is a (classical) topological group.
\end{definition}

\begin{theorem} \label{thm:criterion_xyinv}
With this topology, the usual one-map criterion takes the following form \cite{GoldarRay2022STG}.
Let $F$ be a soft group and $\tau$ a soft topology on $F$. Then $(F,\tau)$ is a soft topological group if and only if the mapping
\begin{equation}
\Delta:\SE(F)\times \SE(F)\to \SE(F),\qquad \Delta(a,b)=a\ast b^{-1},
\end{equation}
is continuous when $\SE(F)\times \SE(F)$ carries the product topology $\tau^\ast\times \tau^\ast$.
\end{theorem}
\begin{proof}
Assume first that $(F,\tau)$ is a soft topological group. Then $(\SE(F),\tau^\ast)$ is a
topological group. In particular, inversion $\iota(b)=b^{-1}$ and multiplication
$m(a,b)=a\ast b$ are continuous. Hence
\begin{equation}
\Delta(a,b)=a\ast b^{-1}=m\bigl(a,\iota(b)\bigr)
\end{equation}
is a composition of continuous maps, so $\Delta$ is continuous.

Conversely, suppose $\Delta:\SE(F)\times\SE(F)\to \SE(F)$ defined by
$\Delta(a,b)=a\ast b^{-1}$ is continuous. Let $e_F$ be the identity of $\SE(F)$.
Define inversion by
\begin{equation}
\iota(b)=b^{-1}=\Delta(e_F,b).
\end{equation}
Since the map $b\mapsto(e_F,b)$ is continuous and $\Delta$ is continuous, $\iota$ is
continuous. Now define multiplication by
\begin{equation}
m(a,b)=a\ast b=\Delta\bigl(a,\iota(b)\bigr).
\end{equation}
The map $(a,b)\mapsto (a,\iota(b))$ is continuous, hence $m$ is continuous as a composition.
Therefore $(\SE(F),\tau^\ast)$ is a topological group, so $(F,\tau)$ is a soft topological group.
\end{proof}

\begin{remark}\label{rem:translations_soft_homeo}
If $(F,\tau)$ is a soft topological group, then left and right translations and inner automorphisms are homeomorphisms of $(\SE(F),\tau^\ast)$.
\end{remark}

\begin{definition}[Soft bitopological space]\label{def:soft_bitop_space}
A \emph{soft bitopological space} is a triple $(F,\tau_1,\tau_2)$ where $F$ is a soft set (with nonempty sections) and $\tau_1,\tau_2$ are soft topologies on $F$.
\end{definition}

\begin{proposition}\label{thm:induced_bitopology}
Let $(F,\tau_1,\tau_2)$ be a soft bitopological space. Then
$(\SE(F),(\tau_1)^\ast,(\tau_2)^\ast)$ is a classical bitopological space.
\end{proposition}

\section{Soft Bitopological Groups}

\begin{definition}[Soft bitopological group]\label{def:soft_bitop_group}
Let $F$ be a soft group over a group $G$, and let $\tau_1,\tau_2$ be soft topologies on $F$. The triple $(F,\tau_1,\tau_2)$ is called a \emph{soft bitopological group} if 
\begin{equation}
(\SE(F),\ast,(\tau_1)^\ast)\quad\text{and}\quad (\SE(F),\ast,(\tau_2)^\ast)
\end{equation}
are (classical) topological groups. Equivalently, the induced bitopological space
\begin{equation}
(\SE(F),(\tau_1)^\ast,(\tau_2)^\ast)
\end{equation}
is a \emph{bitopological group} in the independent-topologies sense.
\end{definition}

\begin{remark}\label{rem:independent_vs_conjugate}
Definition~\ref{def:soft_bitop_group} treats the two topologies as independent. This is useful for transfer results, but it does not say why the two topologies should be related. Section~\ref{sec:conjugate} treats the conjugate case $(\tau,\tau^{-1})$, which comes from a soft paratopological group and shows when inversion fails to be continuous.
\end{remark}

\begin{theorem}\label{thm:soft_bitop_char}
Let $F$ be a soft group and $\tau_1,\tau_2$ soft topologies on $F$. The following are equivalent:
\begin{romanlist}
\item $(F,\tau_1,\tau_2)$ is a soft bitopological group.
\item For each $i\in\{1,2\}$, the map $\Delta_i:\SE(F)\times\SE(F)\to \SE(F)$ defined by $\Delta_i(a,b)=a\ast b^{-1}$ is continuous when $\SE(F)\times\SE(F)$ carries $(\tau_i)^\ast\times (\tau_i)^\ast$ and $\SE(F)$ carries $(\tau_i)^\ast$.
\item For each $t\in A$, the subgroup $F(t)$ equipped with the two component topologies $(\tau_1)_t$ and $(\tau_2)_t$ is a (classical) bitopological group.
\end{romanlist}
\end{theorem}
\begin{proof}
\emph{(i)$\Leftrightarrow$(ii)} follows by applying Theorem~\ref{thm:criterion_xyinv} separately to $(F,\tau_1)$ and $(F,\tau_2)$.

\emph{(i)$\Rightarrow$(iii).} Fix $t\in A$ and choose a base selector $c\in\SE(F)$. The coordinate projection $\pi_t:\SE(F)\to F(t)$ is continuous for each box-induced topology. Also the coordinate insertion
\begin{equation}
j_t:F(t)\to\SE(F),\qquad
j_t(x)(s)=\begin{cases}
 x,&s=t,\\
 c(s),&s\neq t,
\end{cases}
\end{equation}
is continuous: the inverse image of a box is either empty or an open set in the $t$-coordinate. Therefore the component map $(x,y)\mapsto xy^{-1}$ is
\begin{equation}
F(t)\times F(t)\xrightarrow{j_t\times j_t}\SE(F)\times\SE(F)\xrightarrow{(a,b)\mapsto a\ast b^{-1}}\SE(F)\xrightarrow{\pi_t}F(t),
\end{equation}
and is continuous in each component topology. Hence each $(F(t),(\tau_1)_t,(\tau_2)_t)$ is a classical bitopological group.

\emph{(iii)$\Rightarrow$(i).} Assume each $(F(t),(\tau_1)_t)$ and $(F(t),(\tau_2)_t)$ is a topological group. Fix $i\in\{1,2\}$. It is enough to test continuity of $\Delta(a,b)=a\ast b^{-1}$ on basic boxes in $(\tau_i)^\ast$. Let $B=\Box_{t\in A}U_t$ with $U_t\in(\tau_i)_t$. Then
\begin{equation}
\Delta^{-1}(B)=\left\{(a,b): a(t)b(t)^{-1}\in U_t\ \forall t\in A\right\}
=\Box_{t\in A}\Delta_t^{-1}(U_t),
\end{equation}
where $\Delta_t(x,y)=xy^{-1}$. Each $\Delta_t^{-1}(U_t)$ is open in $F(t)\times F(t)$. Under the natural identification
$\SE(F)\times\SE(F)\cong\prod_{t\in A}(F(t)\times F(t))$, the product topology $(\tau_i)^\ast\times(\tau_i)^\ast$ is the corresponding box topology. Hence $\Delta^{-1}(B)$ is open. Thus $\Delta$ is continuous in $(\tau_i)^\ast$; by Theorem~\ref{thm:criterion_xyinv}, $(F,\tau_i)$ is a soft topological group. Since $i$ was arbitrary, $(F,\tau_1,\tau_2)$ is a soft bitopological group.
\end{proof}

Let $(F,\tau_1,\tau_2)$ be a soft bitopological group. For $a\in \SE(F)$ define left and right translations on $\SE(F)$:
\begin{equation}
L_a(x)=a\ast x,\qquad R_a(x)=x\ast a.
\end{equation}
\begin{proposition}\label{prop:translations}
For each $i\in\{1,2\}$ and each $a\in \SE(F)$, the maps $L_a$ and $R_a$ are homeomorphisms of $(\SE(F),(\tau_i)^\ast)$. Inversion $x\mapsto x^{-1}$ is also a homeomorphism in each $(\tau_i)^\ast$.
\end{proposition}

\begin{proof}
Fix $i$. Since $(\SE(F),(\tau_i)^\ast)$ is a topological group, multiplication and inversion are continuous. Each translation is continuous and has continuous inverse translation by $a^{-1}$. Inversion is its own inverse.
\end{proof}

\section{Conjugate Soft Paratopological Groups}\label{sec:conjugate}

The preceding definition allows two independent group topologies. A natural related second topology comes from a paratopological group: the conjugate topology. We now write this construction in soft-element language.

\begin{definition}[Soft semitopological and paratopological group]\label{def:soft_para}
Let $F$ be a soft group and let $\tau$ be a soft topology on $F$.
\begin{alphlist}
\item $(F,\tau)$ is a \emph{soft semitopological group} if, for every $a\in\SE(F)$, the left and right translations $x\mapsto a\ast x$ and $x\mapsto x\ast a$ are continuous on $(\SE(F),\tau^\ast)$.
\item $(F,\tau)$ is a \emph{soft paratopological group} if the multiplication map
\begin{equation}
m:\SE(F)\times\SE(F)\to\SE(F),\qquad m(a,b)=a\ast b,
\end{equation}
is continuous from $\tau^\ast\times\tau^\ast$ to $\tau^\ast$.
\item $(F,\tau)$ is a \emph{soft topological group} if it is a soft paratopological group and inversion $a\mapsto a^{-1}$ is continuous on $(\SE(F),\tau^\ast)$.
\end{alphlist}
\end{definition}

\begin{definition}[Conjugate soft topology]\label{def:conjugate_soft_top}
For a soft subset $H\sse F$ define $H^{-1}\sse F$ by
\begin{equation}
H^{-1}(t):=\{x^{-1}:x\in H(t)\}\qquad(t\in A).
\end{equation}
The \emph{conjugate} or \emph{inverse} soft topology of $\tau$ is the canonical soft topology $\tau^{-1}$ whose component topology at $t$ is
\begin{equation}
(\tau^{-1})_t:=\{U^{-1}:U\in \tau_t\}.
\end{equation}
Equivalently, $H\in\tau^{-1}$ if and only if $H^{-1}\in\tau_{\mathrm{can}}$. Hence $(\tau^{-1})^{-1}=\tau_{\mathrm{can}}$, so the operation is involutive on canonical soft topologies.
\end{definition}

\begin{proposition}\label{prop:selector_inverse_topology}
For every soft topology $\tau$ on a soft group $F$,
\begin{equation}
(\tau^{-1})^\ast=(\tau^\ast)^{-1},
\end{equation}
where the right-hand side denotes the inverse topology on the group $\SE(F)$: a set $W\subseteq\SE(F)$ is open in $(\tau^\ast)^{-1}$ if and only if $W^{-1}:=\{w^{-1}:w\in W\}$ is open in $\tau^\ast$.
\end{proposition}

\begin{proof}
If $B=\Box_{t\in A}U_t$ is a basic box for $\tau^\ast$, then
\begin{equation}
B^{-1}=\Box_{t\in A}U_t^{-1},
\end{equation}
which is a basic box for $(\tau^{-1})^\ast$. The same calculation in the other direction and passage to arbitrary unions gives the equality of topologies.
\end{proof}

\begin{theorem}\label{thm:conjugate_bitopology}
Let $(F,\tau)$ be a soft paratopological group. Then $(F,\tau^{-1})$ is also a soft paratopological group. Multiplication is continuous in both selector topologies, and inversion is a homeomorphism
\begin{equation}
(\SE(F),\tau^\ast)\longrightarrow(\SE(F),(\tau^{-1})^\ast).
\end{equation}
We call $(\SE(F),\tau^\ast,(\tau^{-1})^\ast)$ the \emph{conjugate bitopological paratopological group} determined by $(F,\tau)$. Moreover, the following conditions are equivalent:
\begin{romanlist}
\item $(F,\tau)$ is a soft topological group;
\item $(F,\tau,\tau^{-1})$ is a soft bitopological group in the sense of Definition~\ref{def:soft_bitop_group};
\item $\tau^\ast=(\tau^{-1})^\ast$.
\end{romanlist}
If $\tau$ is canonical, condition~(iii) is also equivalent to $\tau=\tau^{-1}$.
\end{theorem}

\begin{proof}
Proposition~\ref{prop:selector_inverse_topology} shows that inversion is a homeomorphism from $\tau^\ast$ onto $(\tau^{-1})^\ast$. If multiplication is continuous for $\tau^\ast$, then it is continuous for $(\tau^{-1})^\ast$ because
\begin{equation}
xy=\iota\bigl(\iota(y)\iota(x)\bigr),\qquad \iota(x)=x^{-1}.
\end{equation}
Thus multiplication is continuous in both topologies and inversion interchanges them. The selector group is a topological group in $\tau^\ast$ exactly when inversion is continuous as a self-map. This happens exactly when the topology agrees with its inverse topology. The equivalence with the independent-topologies definition follows at once. Suppose now that $\tau$ is canonical and the two selector topologies are equal. For $U\in\tau_t$, the cylinder $\pi_t^{-1}(U)$ is open in $(\tau^{-1})^\ast$. Fixing the other coordinates shows that $U$ is a union of sets open in $\tau_t^{-1}$. Hence $\tau_t\subseteq\tau_t^{-1}$, and the reverse inclusion is the same. Thus the component topologies agree, and canonicity gives $\tau=\tau^{-1}$.
\end{proof}

\begin{remark}\label{rem:sorgenfrey_context}
The conjugate pair need not be a soft bitopological group under Definition~\ref{def:soft_bitop_group}. It becomes one precisely when inversion is continuous. The Sorgenfrey line under addition gives the standard example: addition is jointly continuous in the lower-limit topology, while inversion changes it to the upper-limit topology.
\end{remark}

\begin{definition}[Soft $\omega$-boundedness for the conjugate join]\label{def:soft_omega_bounded}
Let $(F,\tau)$ be a soft paratopological group. We say that the conjugate join is \emph{soft $\omega$-bounded} if every countable subset of $\SE(F)$ has compact closure in the join topology
\begin{equation}
\tau^\ast\vee(\tau^{-1})^\ast.
\end{equation}
This is the $\omega$-boundedness condition for the join topology used by Alas and Sanchis.
\end{definition}

\begin{theorem}\label{thm:soft_automatic_continuity}
The classical results below apply directly to the selector group.
\begin{romanlist}
\item If $(F,\tau)$ is a soft semitopological group and $(\SE(F),\tau^\ast)$ is locally compact Hausdorff, then $(F,\tau)$ is a soft topological group (Ellis \cite{Ellis1957}).
\item If $(F,\tau)$ is a soft paratopological group and $(\SE(F),\tau^\ast)$ is compact and $T_0$, then $(F,\tau)$ is a soft topological group (Numakura \cite{Numakura1952}; Romaguera--Sanchis \cite{RomagueraSanchis2007}).
\item If $(F,\tau)$ is a soft paratopological group and $(\SE(F),\tau^\ast)$ is countably compact and $T_0$, then $(F,\tau)$ is a soft topological group if and only if the conjugate join $\tau^\ast\vee(\tau^{-1})^\ast$ is soft $\omega$-bounded (Alas--Sanchis \cite{AlasSanchis2007}).
\end{romanlist}
\end{theorem}

\begin{proof}
Apply the cited classical theorem to the induced semitopological or paratopological group $(\SE(F),\tau^\ast)$. By Proposition~\ref{prop:selector_inverse_topology}, the conjugate topology on this group is $(\tau^{-1})^\ast$. Thus the Alas--Sanchis condition is the soft $\omega$-boundedness condition in Definition~\ref{def:soft_omega_bounded}.
\end{proof}

\begin{remark}\label{rem:alas_sanchis_finite_parameters}
Item~(iii) should be read with Proposition~\ref{prop:box_join_omega}. For infinite parameter sets, $\tau^\ast$ is a box topology, so countable compactness and $\omega$-boundedness are strong assumptions. The criterion is therefore most useful for finite $A$, where box and product topologies agree, and for special infinite-parameter cases where the box topology still has compactness properties. Arguments using the coarser product topology $\tau^\Pi$ need separate hypotheses.
\end{remark}

\begin{theorem}\label{thm:ravsky_soft}
Let $(F,\tau)$ be a soft paratopological group and let $e_F$ be the identity selector. Then $(F,\tau)$ is a soft topological group if and only if, for every $\tau^\ast$-neighbourhood $U$ of $e_F$, the set
\begin{equation}
U\cap U^{-1}
\end{equation}
has nonempty interior in $(\SE(F),\tau^\ast)$.
\end{theorem}

\begin{proof}
Apply the Ravsky--Reznichenko criterion to the paratopological group $(\SE(F),\tau^\ast)$. Lemma~7.1 of Banakh--Ravsky~\cite{BanakhRavsky2019Banalytic} states this criterion and traces it to Lemmas~5.1, 5.9 and~5.10 of Ravsky's thesis~\cite{RavskyThesis2002} and Proposition~4 of Ravsky--Reznichenko~\cite{RavskyReznichenko2002}.
\end{proof}

\begin{proposition}\label{prop:box_join_omega}
The join topology
\begin{equation}
\tau^\ast\vee(\tau^{-1})^\ast
\end{equation}
on $\SE(F)=\prod_{t\in A}F(t)$ is the box topology generated by the component joins $\tau_t\vee\tau_t^{-1}$. If $A$ is finite and each coordinate join $(F(t),\tau_t\vee\tau_t^{-1})$ is $\omega$-bounded, then the soft conjugate join is $\omega$-bounded. For infinite $A$, coordinatewise compactness or coordinatewise $\omega$-boundedness need not imply the corresponding soft property for the box join.
\end{proposition}

\begin{proof}
Every finite intersection of boxes from $\tau^\ast$ and $(\tau^{-1})^\ast$ is open in the box topology generated by the component joins. Conversely, let $x$ lie in a box $\prod_tW_t$ with $W_t\in\tau_t\vee\tau_t^{-1}$. For each $t$, choose $U_t\in\tau_t$ and $V_t\in\tau_t^{-1}$ such that
\begin{equation}
x(t)\in U_t\cap V_t\subseteq W_t.
\end{equation}
Then $(\prod_tU_t)\cap(\prod_tV_t)$ is a join-open neighbourhood of $x$ contained in $\prod_tW_t$. This proves the box-join identity. If $A$ is finite and $C\subseteq\SE(F)$ is countable, let $K_t$ be the compact closure of the $t$-th coordinate projection of $C$. The closure of $C$ is contained in the finite product $\prod_{t\in A}K_t$. It is closed in that compact product and is therefore compact.

For the infinite-parameter case, take $F(t)=\mathbb Z/2\mathbb Z$ for every $t\in A$ with the discrete group topology in each coordinate. Each coordinate is compact and $\omega$-bounded. But the box topology on $\prod_{t\in A}F(t)$ is discrete, because every singleton is a box. If $A$ is infinite, this selector group is infinite and an infinite countable subset has noncompact closure. Hence the soft box join is not $\omega$-bounded.
\end{proof}

\section{Pairwise Selector Separation Axioms}\label{sec:separation}

Pairwise separation through soft elements was considered in \cite{SoftBTSviaSE}. The Hausdorff clause used there requires $H\Sinter K=\Phi$, which is too strong when two selectors agree at some parameters. To match the induced bitopology, disjointness is taken in the selector space.

\begin{definition}[Pairwise soft separation axioms]\label{def:pairwise_T}
Let $(F,\tau_1,\tau_2)$ be a soft bitopological space.
\begin{alphlist}
\item It is \emph{pairwise soft $T_0$} if for every $a,b\in\SE(F)$ with $a\neq_s b$ there is $H\in\tau_1\cup\tau_2$ such that $a\in_sH$ and $b\notin_sH$, or vice versa.
\item It is \emph{pairwise soft $T_1$} if for every $a\neq_s b$ there are $H\in\tau_1$ and $K\in\tau_2$ such that $a\in_sH$, $b\notin_sH$, $b\in_sK$, and $a\notin_sK$.
\item It is \emph{pairwise soft Hausdorff}, or \emph{pairwise selector-Hausdorff}, if for every $a\neq_s b$ there are $H\in\tau_1$ and $K\in\tau_2$ such that
\begin{equation}
a\in_sH,\qquad b\in_sK,\qquad \SE(H)\cap\SE(K)=\varnothing.
\end{equation}
\end{alphlist}
The last condition says that the selector boxes are disjoint. It does not require $H(t)\cap K(t)=\varnothing$ for every parameter.
\end{definition}

\begin{proposition}\label{prop:T_chain}
Pairwise soft Hausdorff implies pairwise soft $T_1$, and pairwise soft $T_1$ implies pairwise soft $T_0$.
\end{proposition}

\begin{proof}
If $a\in_sH$, $b\in_sK$, and $\SE(H)\cap\SE(K)=\varnothing$, then $b\notin_sH$ and $a\notin_sK$. This gives the $T_1$ condition. The implication from $T_1$ to $T_0$ is immediate.
\end{proof}

\begin{theorem}\label{thm:componentwise_T}
Let $(F,\tau_1,\tau_2)$ be a soft bitopological space and let $j\in\{0,1,2\}$, where $j=2$ means pairwise selector-Hausdorff. For $t\in A$, write
\begin{equation}
\mathcal F_t:=\bigl(F(t),(\tau_1)_t,(\tau_2)_t\bigr).
\end{equation}
\begin{romanlist}
\item If $(F,\tau_1,\tau_2)$ is pairwise soft $T_j$, then every $\mathcal F_t$ is pairwise $T_j$.
\item If both $\tau_1$ and $\tau_2$ are canonical and every $\mathcal F_t$ is pairwise $T_j$, then $(F,\tau_1,\tau_2)$ is pairwise soft $T_j$.
\end{romanlist}
Hence, for canonical soft bitopologies,
\begin{equation}
(F,\tau_1,\tau_2)\text{ is pairwise soft }T_j
\quad\Longleftrightarrow\quad
(F(t),(\tau_1)_t,(\tau_2)_t)\text{ is pairwise }T_j\text{ for every }t\in A.
\end{equation}
\end{theorem}

\begin{proof}
\emph{(i)} Fix $t\in A$ and distinct $x,y\in F(t)$. Choose $c\in\SE(F)$ and define selectors $a,b$ by
\begin{equation}
a(t)=x,\quad b(t)=y,\quad a(s)=b(s)=c(s)\quad(s\neq t).
\end{equation}
For $j=0$, choose $H\in\tau_1\cup\tau_2$ with, say, $a\in_sH$ and $b\notin_sH$. Since $a$ and $b$ agree outside $t$, the failure of $b$ to belong to $H$ must occur at $t$. Thus $x\in H(t)$ and $y\notin H(t)$.

For $j=1$, choose $H\in\tau_1$ and $K\in\tau_2$ as in Definition~\ref{def:pairwise_T}. Again, because $a$ and $b$ agree outside $t$, their nonmembership must occur at $t$. Hence
\begin{equation}
x\in H(t),\quad y\notin H(t),\quad y\in K(t),\quad x\notin K(t).
\end{equation}
For $j=2$, choose $H\in\tau_1$ and $K\in\tau_2$ with $a\in_sH$, $b\in_sK$, and $\SE(H)\cap\SE(K)=\varnothing$. If $z\in H(t)\cap K(t)$, define $d(t)=z$ and $d(s)=c(s)$ for $s\neq t$. Since $a(s)=b(s)=c(s)$ outside $t$, this selector lies in both $\SE(H)$ and $\SE(K)$, a contradiction. Thus $H(t)\cap K(t)=\varnothing$.

\emph{(ii)} Let $a,b\in\SE(F)$ be distinct and choose $t_0$ with $a(t_0)\neq b(t_0)$. For $j=0$, take a component open set $U$ that separates $a(t_0)$ and $b(t_0)$, and define $H(t_0)=U$ and $H(t)=F(t)$ for $t\neq t_0$. Canonicity gives $H\in\tau_1\cup\tau_2$.

For $j=1$, choose $U\in(\tau_1)_{t_0}$ and $V\in(\tau_2)_{t_0}$ with
\begin{equation}
a(t_0)\in U,\quad b(t_0)\notin U,\quad b(t_0)\in V,\quad a(t_0)\notin V.
\end{equation}
Set $H(t_0)=U$, $K(t_0)=V$, and use the full sections $F(t)$ elsewhere. Then $H\in\tau_1$ and $K\in\tau_2$, and they give the soft $T_1$ condition.

For $j=2$, choose disjoint $U\in(\tau_1)_{t_0}$ and $V\in(\tau_2)_{t_0}$ containing $a(t_0)$ and $b(t_0)$, respectively. Define $H$ and $K$ as above. If a selector belonged to both $\SE(H)$ and $\SE(K)$, its value at $t_0$ would lie in $U\cap V$, which is impossible. Thus the two selector boxes are disjoint.
\end{proof}

\begin{remark}
For canonical soft bitopologies, Theorem~\ref{thm:componentwise_T} reduces pairwise separation to the component spaces. Sectionwise disjointness $H\Sinter K=\Phi$ is not needed.
\end{remark}

\begin{proposition}\label{prop:soft_to_induced_T}
If $(F,\tau_1,\tau_2)$ is pairwise soft $T_j$ for $j\in\{0,1,2\}$, then the induced bitopological space $(\SE(F),(\tau_1)^\ast,(\tau_2)^\ast)$ is pairwise $T_j$ in the classical sense. The converse can fail, already for $T_0$.
\end{proposition}

\begin{proof}
For every $H\in\tau_i$, the selector box $\SE(H)$ is open in $(\tau_i)^\ast$. The three implications therefore follow directly from Definition~\ref{def:pairwise_T}. For the converse, take the topology $\tau$ in Remark~\ref{rem:selector_forgets_correlations} and put $\tau_1=\tau_2=\tau$. Both induced topologies are discrete, but the selectors $(0,1)$ and $(1,0)$ cannot be separated by the original proper soft open sets. Thus the soft space is not pairwise soft $T_0$.
\end{proof}

A $T_0$ topological group is automatically Hausdorff. We use this standard fact below.

\begin{proposition}\label{prop:T0_implies_T2_group}
Let $(F,\tau_1,\tau_2)$ be a soft bitopological group. If for some $i\in\{1,2\}$ the induced topology $(\tau_i)^\ast$ on $\SE(F)$ is $T_0$, then $(\SE(F),(\tau_i)^\ast)$ is Hausdorff.
\end{proposition}

\begin{proof}
$(\SE(F),(\tau_i)^\ast)$ is a topological group. The classical theorem for topological groups yields $T_0\Rightarrow T_2$.
\end{proof}

\begin{remark}
Pairwise soft $T_0$ does not force $(\tau_1)^\ast$ and $(\tau_2)^\ast$ to be $T_0$ individually, because the separating open set may come from either topology. In contrast, pairwise soft $T_1$ implies that each topology separates points in the $T_0$ sense: for $a\neq_s b$ the definition provides a $\tau_1$-open separating $a$ from $b$ and a $\tau_2$-open separating $b$ from $a$.
\end{remark}

\section{Pairwise Soft Compactness and Connectedness}

Throughout this section $(F,\tau_1,\tau_2)$ denotes a soft bitopological \emph{group}.
Unless otherwise stated, $H\sse F$ is an arbitrary soft subset.

\subsection*{Pairwise soft compactness}

In this paper, pairwise compactness means compactness in the join topology $\alpha\vee\beta$. Since $\alpha\cup\beta$ is a subbase for the join, this is equivalent to requiring a finite subcover for every cover by members of $\alpha\cup\beta$. In the soft setting we use selector covers, because a parameterwise soft union need not represent the union of selector sets (Theorem~\ref{thm:selector_union_exact}).

\begin{definition}[Pairwise soft compactness]\label{def:pairwise_compact}
Let $(F,\tau_1,\tau_2)$ be a soft bitopological space and let $H\sse F$ have nonempty sections. The soft subset $H$ is \emph{pairwise soft compact} if $\SE(H)$ is compact as a subspace of
\begin{equation}
(\tau_1)^\ast\vee(\tau_2)^\ast.
\end{equation}
Equivalently, every cover of $\SE(H)$ by sets open in $(\tau_1)^\ast$ or in $(\tau_2)^\ast$ has a finite subcover. If $H=F$, we call $(F,\tau_1,\tau_2)$ pairwise soft compact.
\end{definition}

\begin{theorem}\label{thm:soft_to_component_compact}
Let $(F,\tau_1,\tau_2)$ be a soft bitopological space and let $H\sse F$ have nonempty sections. If $H$ is pairwise soft compact, then $H(t)$ is pairwise compact in $(F(t),(\tau_1)_t,(\tau_2)_t)$ for every $t\in A$.
\end{theorem}

\begin{proof}
For each $t$, the coordinate projection
\begin{equation}
\pi_t:\bigl(\SE(F),(\tau_1)^\ast\vee(\tau_2)^\ast\bigr)
\longrightarrow\bigl(F(t),(\tau_1)_t\vee(\tau_2)_t\bigr)
\end{equation}
is continuous. Its restriction maps $\SE(H)$ onto $H(t)$. Thus $H(t)$ is a continuous image of a compact space and is compact in the component join topology.
\end{proof}

\begin{theorem}\label{thm:finiteness_principle}
Assume that $A$ is finite, let $(F,\tau_1,\tau_2)$ be a soft bitopological space, and let $H\sse F$ have nonempty sections. If for every $t\in A$ the subset $H(t)$ is pairwise compact in $(F(t),(\tau_1)_t,(\tau_2)_t)$, then $H$ is pairwise soft compact in $(F,\tau_1,\tau_2)$.
\end{theorem}

\begin{proof}
For finite $A$, the box topology and the Tychonoff product topology coincide. Moreover,
\begin{equation}
(\tau_1)^\ast\vee(\tau_2)^\ast
=
\prod_{t\in A}\bigl((\tau_1)_t\vee(\tau_2)_t\bigr)
\end{equation}
as a finite product topology on $\prod_{t\in A}F(t)$. Since each $H(t)$ is compact in the component join, the finite product $\prod_{t\in A}H(t)=\SE(H)$ is compact. Thus $H$ is pairwise soft compact.
\end{proof}

\begin{proposition}\label{prop:infinite_A_compact_fail}
When $A$ is infinite, coordinatewise pairwise compactness need not imply pairwise soft compactness, even for canonical soft bitopologies and even for compact Hausdorff coordinate groups.
\end{proposition}

\begin{proof}
Let $F(t)=\mathbb Z/2\mathbb Z$ for every $t\in A$ and give each coordinate the discrete topology; set $\tau_1=\tau_2$ to be the corresponding canonical soft topology. Each coordinate is a compact Hausdorff topological group. However, the induced selector topology on $\SE(F)=\prod_{t\in A}F(t)$ is the box topology. Since each singleton $\prod_{t\in A}\{x_t\}$ is a box, this topology is discrete. If $A$ is infinite, $\SE(F)$ is infinite, and an infinite discrete space is not compact. Hence $F$ is not pairwise soft compact.
\end{proof}

\begin{remark}\label{rem:tychonoff_soft_compact}
The same example is compact in the Tychonoff topology $\tau^\Pi$ by Tychonoff's theorem. Thus the finite/infinite split does not come from the coordinates; it comes from using the box topology in the soft-element construction.
\end{remark}

\begin{proposition}\label{prop:compact_closed}
Let $(F,\tau_1,\tau_2)$ be a soft bitopological group. Fix $i\in\{1,2\}$ and assume that $(\SE(F),(\tau_i)^\ast)$ is Hausdorff. If $H\sse F$ is a soft subgroup such that $\SE(H)$ is compact in $(\SE(F),(\tau_i)^\ast)$, then $\SE(H)$ is closed in $(\SE(F),(\tau_i)^\ast)$. This is equivalent to $H(t)$ being closed in $(F(t),(\tau_i)_t)$ for every $t\in A$.
\end{proposition}

\begin{proof}
In any Hausdorff topological space, compact subsets are closed. Hence $\SE(H)$ is closed. The rectangular subset $\SE(H)=\prod_t H(t)$ is closed in the box topology if and only if each $H(t)$ is closed in its component: one direction follows by taking coordinate insertions, and the other follows because
$\SE(F)\setminus\SE(H)=\bigcup_{t\in A}\pi_t^{-1}(F(t)\setminus H(t))$.
\end{proof}
\subsection*{Bi-soft connectedness}
Connectedness is most naturally expressed on the induced topological groups.

\begin{definition}\label{def:bi_connected}
A soft bitopological space $(F,\tau_1,\tau_2)$ is called \emph{bi-soft connected} if both induced spaces $(\SE(F),(\tau_1)^\ast)$ and $(\SE(F),(\tau_2)^\ast)$ are connected in the classical sense.
\end{definition}

\begin{proposition}\label{prop:connected_clopen}
If $(F,\tau_1,\tau_2)$ is bi-soft connected, then for each $i\in\{1,2\}$ the only clopen subsets of $(\SE(F),(\tau_i)^\ast)$ are $\varnothing$ and $\SE(F)$.
\end{proposition}

\begin{proof}
A topological space is connected if and only if it has no non-trivial clopen subsets.
Apply this to each induced space $(\SE(F),(\tau_i)^\ast)$.
\end{proof}

\begin{proposition}\label{prop:connected_hom}
Let $(F,\tau_1,\tau_2)$ be a soft bitopological group and fix $i\in\{1,2\}$.
Assume that $(\SE(F),(\tau_i)^\ast)$ is connected.
If $D$ is a discrete topological group and
\begin{equation}
\varphi:(\SE(F),(\tau_i)^\ast)\to D
\end{equation}
is a continuous group homomorphism, then $\varphi$ is constant.
\end{proposition}

\begin{proof}
The image of a connected space under a continuous map is connected.
Every connected subset of a discrete space is a singleton, hence $\varphi(\SE(F))$ is a singleton.
\end{proof}

\begin{remark}\label{rem:identity_component}
For each $i\in\{1,2\}$, the identity component of the topological group $(\SE(F),(\tau_i)^\ast)$ is a closed normal subgroup.
Thus bi-soft connectedness is equivalent to saying that the identity component is the whole group for both induced topologies.
\end{remark}
\section{Homomorphisms of Soft Bitopological Groups}
\begin{definition}[Soft bitopological group homomorphism]\label{def:hom}
Let $(F,\tau_1,\tau_2)$ and $(H,\sigma_1,\sigma_2)$ be soft bitopological groups (over groups $G$ and $G'$). A map
\begin{equation}
\varphi:\SE(F)\to \SE(H)
\end{equation}
is called a \emph{soft bitopological group homomorphism} if:
\begin{romanlist}
\item $\varphi$ is a group homomorphism $(\SE(F),\ast)\to(\SE(H),\ast)$;
\item for each $i\in\{1,2\}$, $\varphi:(\SE(F),(\tau_i)^\ast)\to(\SE(H),(\sigma_i)^\ast)$ is continuous.
\end{romanlist}
If $\varphi$ is bijective and its inverse also satisfies (i)--(ii), then $\varphi$ is a \emph{soft bitopological isomorphism}.
\end{definition}

\begin{remark}
In many soft applications, one restricts to homomorphisms that act \emph{parameterwise}: for each $t\in A$ there is a homomorphism $\varphi_t:F(t)\to H(t)$ such that
\begin{equation}
(\varphi(a))(t)=\varphi_t(a(t))\qquad (a\in \SE(F)).
\end{equation}
Not every abstract homomorphism $\SE(F)\to \SE(H)$ decomposes this way; therefore the choice of the morphism notion should be stated explicitly in applications.
\end{remark}

\begin{proposition}\label{prop:kernel_image}
Let $\varphi:(F,\tau_1,\tau_2)\to(H,\sigma_1,\sigma_2)$ be a soft bitopological group homomorphism.
Then $\ker(\varphi)$ is a normal subgroup of $\SE(F)$ and $\mathrm{im}(\varphi)$ is a subgroup of $\SE(H)$. If $\varphi$ is surjective and $(F,\tau_1,\tau_2)$ is pairwise soft compact (resp.\ bi-soft connected), then $(H,\sigma_1,\sigma_2)$ is pairwise soft compact (resp.\ bi-soft connected).
\end{proposition}

\begin{proof}
The subgroup statements are classical. For pairwise compactness, the inverse image of every $(\sigma_1)^\ast$-open or $(\sigma_2)^\ast$-open set is open in the corresponding domain topology. Hence $\varphi$ is continuous from $(\tau_1)^\ast\vee(\tau_2)^\ast$ to $(\sigma_1)^\ast\vee(\sigma_2)^\ast$, and compactness passes to the surjective image. For bi-soft connectedness, each map
\begin{equation}
\varphi:(\SE(F),(\tau_i)^\ast)\longrightarrow(\SE(H),(\sigma_i)^\ast)
\end{equation}
is continuous and onto, so each target space is connected.
\end{proof}

\section{Examples}

\begin{example}\label{ex:disc_indisc}
Let $G$ be any group and let $F(t)=G$ for all $t\in A$. Then $F$ is a soft group and $\SE(F)=G^A$. 
Let $\tau_{\mathrm{disc}}$ be the soft discrete topology (all soft subsets are open) and $\tau_{\mathrm{ind}}=\{\Phi,F\}$ the soft indiscrete topology. 
Then $(F,\tau_{\mathrm{disc}},\tau_{\mathrm{ind}})$ is a soft bitopological group. The first induced topology is discrete and Hausdorff, while the second is indiscrete and not even $T_0$.
\end{example}

\begin{example}\label{ex:D4_example}
Let $G=D_{8}=\langle r,s\mid r^{4}=e,\ s^{2}=e,\ srs=r^{-1}\rangle$ and let $A=\{t_1,t_2\}$. Define
\begin{equation}
F(t_1)=\langle r\rangle,\qquad F(t_2)=D_8.
\end{equation}
Let $F_1,F_2\sse F$ be given by
\begin{equation}
F_1(t_1)=\langle r^2\rangle,\qquad F_1(t_2)=\langle r\rangle,
\end{equation}
\begin{equation}
F_2(t_1)=r\langle r^2\rangle,\qquad F_2(t_2)=s\langle r\rangle.
\end{equation}
Then
\begin{equation}
\tau_1=\{\Phi,F,F_1,F_2\}
\end{equation}
is a soft topology because $F_1\Sunion F_2=F$ and $F_1\Sinter F_2=\Phi$. Its component topologies are
\begin{equation}
(\tau_1)_{t_1}=\{\varnothing,\langle r\rangle,\langle r^2\rangle,r\langle r^2\rangle\},
\end{equation}
\begin{equation}
(\tau_1)_{t_2}=\{\varnothing,D_8,\langle r\rangle,s\langle r\rangle\}.
\end{equation}
Each is the coset topology of an index-two subgroup and is a group topology. Let $\widehat\tau_1$ be the canonical enlargement determined by these two component topologies. Then $\tau_1\subsetneq\widehat\tau_1$; for example, the componentwise choice $\langle r^2\rangle$ at $t_1$ and $s\langle r\rangle$ at $t_2$ belongs to $\widehat\tau_1$ but not to $\tau_1$. Nevertheless,
\begin{equation}
(\tau_1)^\ast=(\widehat\tau_1)^\ast.
\end{equation}
The selector group is therefore a topological group in this common box topology. Under Definition~\ref{def:soft_top_group}, both $(F,\tau_1)$ and $(F,\widehat\tau_1)$ are soft topological groups, although the selector construction cannot distinguish their different parameter correlations.

Let $\tau_2$ be the soft discrete topology on $F$. Then $(F,\tau_1,\tau_2)$ is a soft bitopological group.
\end{example}

\begin{example}\label{ex:padic_qadic}
Let $G=(\mathbb Z,+)$ and let $p\neq q$ be primes. Let $\tau_p$ and $\tau_q$ denote the $p$-adic and $q$-adic group topologies on $\mathbb Z$, whose neighbourhood bases at $0$ are $\{p^n\mathbb Z:n\in\mathbb N\}$ and $\{q^n\mathbb Z:n\in\mathbb N\}$. Both are Hausdorff group topologies. They are incomparable: if $\tau_p\subseteq\tau_q$, then $p\mathbb Z$ is $\tau_q$-open, so some basic $\tau_q$-neighbourhood $q^n\mathbb Z$ is contained in $p\mathbb Z$, forcing $p\mid q^n$, which is impossible; symmetrically $\tau_q\not\subseteq\tau_p$.

Let $A=\{t_1,t_2\}$ and define $F(t_1)=F(t_2)=\mathbb Z$. Let $\tau_1$ be the canonical soft topology with component topologies
\begin{equation}
(\tau_1)_{t_1}=\tau_p,\qquad (\tau_1)_{t_2}=\tau_q,
\end{equation}
and let $\tau_2$ be the canonical soft topology with the two assignments reversed. Then $(F,\tau_1,\tau_2)$ is a soft bitopological group. In this example both induced topologies on $\SE(F)\cong \mathbb Z\times \mathbb Z$ are Hausdorff and different; the difference comes from assigning the incomparable group topologies to the two coordinates in opposite order.
\end{example}

\begin{example}\label{ex:sorgenfrey_soft}
Let $G=(\mathbb R,+)$ and let $\mathbb S$ denote the lower-limit topology on $\mathbb R$, generated by intervals $[a,b)$. Addition is jointly continuous in $(\mathbb R,\mathbb S)$, so this is a paratopological group, but inversion $x\mapsto -x$ is not continuous because the inverse image of $[0,1)$ is $(-1,0]$, which is not lower-limit open. The conjugate topology $\mathbb S^{-1}$ is the upper-limit topology generated by intervals $(a,b]$.

Let $F(t)=\mathbb R$ for every $t\in A$ and let $\tau$ be the canonical soft topology with $\tau_t=\mathbb S$ for all $t$. Then $(F,\tau)$ is a soft paratopological group, and $\tau^{-1}$ is the canonical soft topology with upper-limit components. If $A$ is nonempty, $\tau^\ast\neq(\tau^{-1})^\ast$; hence $(F,\tau)$ is not a soft topological group. The pair $(\tau,\tau^{-1})$ shows the failure of inversion continuity.
\end{example}

\section{Conclusion}
The selector construction equips $\SE(F)=\prod_{t\in A}F(t)$ with the box topology determined by the component topologies. It gives the usual group-theoretic transfer results, but it also has a clear limitation: it may lose correlations contained in a noncanonical soft topology. For a soft paratopological group, the inverse topology forms a conjugate pair with the original topology, and equality of the two selector topologies is equivalent to continuity of inversion. The box topology also explains why compactness and $\omega$-boundedness behave differently for finite and infinite parameter sets. Finally, selectors preserve soft intersections exactly, while soft unions may create mixed selectors.

\end{document}